\documentclass[12pt]{amsart}
\usepackage{amsmath}
\usepackage{amsfonts}
\usepackage{amssymb}
\usepackage{amsthm}
\usepackage[pdfencoding=auto]{hyperref}
\usepackage{cleveref}
\usepackage{thmtools}
\usepackage{thm-restate}
\usepackage{appendix}
\usepackage{enumerate}
\newcommand{\Z}{\mathbb{Z}}
\newcommand{\F}{\mathbb{F}}
\newtheorem{theorem}{Theorem}
\newtheorem{lemma}{Lemma}

\newtheorem{conjecture}{Conjecture}

\theoremstyle{remark}
\newtheorem{remark}{Remark}

\title{Sets with few differences in abelian groups}
\author{Mitchell Lee}
\address{Harvard University}
\email{mitchell@math.harvard.edu}

\begin{document}
\newtheoremstyle{casestyle}
{\topsep}   
{\topsep}   
{\normalfont}  
{0pt}       
{\bfseries} 
{:\newline}         
{5pt plus 1pt minus 1pt} 
{}          
\theoremstyle{casestyle}
\newtheorem{case}{Case}
\begin{abstract}
Let $(G, +)$ be an abelian group. In 2004, Eliahou and Kervaire found an explicit formula for the smallest possible cardinality of the sumset $A+A$, where $A \subseteq G$ has fixed cardinality $r$. We consider instead the smallest possible cardinality of the difference set $A-A$, which is always greater than or equal to the smallest possible cardinality of $A+A$ and can be strictly greater. We conjecture a formula for this quantity, and prove the conjecture in the case that $G$ is a cyclic group or a vector space over a finite field. This resolves a conjecture of Bajnok and Matzke on signed sumsets.
\end{abstract}
\maketitle
\section{Introduction}
Let $G$ be a finite abelian group of order $N$ written with additive notation. Given subsets $A, B \subseteq G$, the \textit{sumset} of $A$ and $B$ is defined as \[A+B = \{a + b \mid a \in A, b \in B\}\] and the \textit{difference set} of $A$ and $B$ is defined as \[A - B = \{a - b \mid a \in A, b \in B\}.\]
Let $-A$ denote the difference set $\{0\} - A = \{-a \mid a \in A\}$.

\par Given integers $r$ and $s$ with $1 \leq r, s \leq N$, define
\begin{align}
	\mu_G(r, s) &= \min \{|A + B| \mid A, B \subseteq G, |A| = r, |B| = s\} \label{mu}\\
	\rho^+_G(r) &= \min \{|A + A| \mid A \subseteq G, |A| = r\} \label{rhoplus} \\
	\rho^-_G(r) &= \min \{|A - A| \mid A \subseteq G, |A| = r\}. \label{rhominus}
\end{align}
We remark that taking $B = A$ in \eqref{mu} yields $\mu_G(r, r) \leq \rho^+_G(r)$ and taking $B = -A$ yields $\mu_G(r, r) \leq \rho^-_G(r)$.
\par The functions $\mu_G(r, s)$ and $\rho^+_G(r)$ have held considerable interest for over 200 years. In 1813, Cauchy \cite{cauchy} proved the following classical result, which was rediscovered by Davenport \cite{davenport} in 1935.
\begin{restatable}[Cauchy-Davenport Theorem \cite{cauchy,davenport}]{theorem}{cauchydavenport}\label{cauchydavenport}
	\par Let $G=\Z / p\Z$ where $p$ is prime. Then $\mu_{G}(r, s) = \min\{r + s - 1, p\}$ for $1 \leq r, s \leq p$.
\end{restatable}
\par In 2004, Eliahou and Kervaire \cite{sumsets_infinite} used a classical result of Kneser \cite{kneser} to compute $\mu_G(r, s)$ and $\rho_G^+(r)$ for all finite abelian groups $G$.

\begin{theorem}[Eliahou and Kervaire, {\cite[Theorem~2, Proposition~7]{sumsets_infinite}}] \label{sumsets_arbitrary}
	Let $G$ be a finite abelian group of order $N$. Then \[\mu_G(r, s) = \min_{d \in D(N)} d \left( \left\lceil\frac{r}{d}\right\rceil + \left\lceil\frac{s}{d}\right\rceil - 1\right)\] for $1 \leq r, s \leq N$, where $D(N)$ denotes the set of positive divisors of $N$. Furthermore, we have $\rho^+_G(r) = \mu_G(r, r)$.
\end{theorem}
\begin{remark}\label{remarkdepend}
	By \autoref{sumsets_arbitrary}, the quantities $\mu_G(r, s)$ and $\rho_G^+(r)$ depend on $N$, $r$, and $s$, but not the group structure of $G$.
\end{remark}

However, there is no known explicit formula for $\rho^-_G(r)$. In \cite{signed_sumsets, signed_sumsets_elementary}, Bajnok and Matzke considered an $h$-fold variant of this problem. A small adaptation of their proofs yields the following upper bound for $\rho^-_G(r)$, which we conjecture holds with equality.

\begin{restatable}[{cf. \cite[Theorem~5]{signed_sumsets}}]{theorem}{upperbound} \label{upperbound}
	Let $G$ be a finite abelian group of order $N$. Let $e = \exp G$ be the exponent of $G$; that is, the least common multiple of the orders of the elements of $G$. For $1 \leq r \leq N$, define \[D(N, e, r) = \{d_1 d_2 \mid d_1 \in D(N/e), d_2 \in D(e), d_1 e \geq r\}.\] Then \[\rho_G^-(r) \leq \min_{d \in D(N, e, r)} d \left(2\left\lceil\frac{r}{d}\right\rceil - 1\right).\]
\end{restatable}

\begin{conjecture}[cf. {\cite[Conjecture~10]{signed_sumsets}}]\label{difference_conjecture}
	The inequality in \autoref{upperbound} holds with equality. That is, under the hypotheses of \autoref{upperbound}, we have \[\rho_G^-(r) = \min_{d \in D(N, e, r)} d \left(2\left\lceil\frac{r}{d}\right\rceil - 1\right).\]
\end{conjecture}

\begin{remark}
	We have the inequality $\rho_G^+(r) = \mu_G(r, r) \leq \rho_G^-(r)$, and it is possible that $\rho_G^+(r) < \rho_G^-(r)$. For example, if $G = (\Z / 3 \Z)^2$, then $\rho_G^+(4) = 7$ and $\rho_G^-(4) = 9$. It is also worth noting that in contrast to $\rho_G^+(r)$ (see \autoref{remarkdepend}), the quantity $\rho_G^-(r)$ cannot be determined from $N$ and $r$ alone.
\end{remark}

The goal of this paper is to prove two important special cases of \autoref{difference_conjecture}.

First, consider the case that $G = \Z / N\Z$ is a finite cyclic group. In this case, we have $e = \exp G = N$, so $D(N, e, r) = D(N)$ for $1 \leq r \leq N$. Thus, the statement of \autoref{difference_conjecture} becomes \autoref{cyclic} below.

\begin{restatable}[{cf. \cite[Theorem~4]{signed_sumsets}}]{theorem}{cyclic} \label{cyclic}
	Let $G = \Z / N\Z$. Then \[\rho_G^-(r) = \min_{d \in D(N)} d \left( 2\left\lceil\frac{r}{d}\right\rceil  - 1\right)\] for $1 \leq r \leq N$.
\end{restatable}

Second, consider the case that $G = (\Z / p\Z)^d$ where $p$ is prime and $d \geq 0$. Then \autoref{mainthm} below, which is the main result of this paper, computes $\rho_G^-(r)$ for $1 \leq r \leq p^d$. We will verify in \autoref{sectionoutline} that \autoref{mainthm} agrees with the prediction given by \autoref{difference_conjecture}.

\begin{restatable}{theorem}{mainthm}\label{mainthm}
	Let $G=(\Z / p\Z)^d$ where $p$ is prime and $d \geq 0$. Let $t$ and $r$ be integers with $0 \leq t \leq d$ and $p^t < r \leq p^{t+1}$. Then \[\rho_G^-(r) = p^t \min\left\{2\left\lceil\frac{r}{p^t}\right\rceil - 1, p\right\}.\]
\end{restatable}

\par As a consequence of \autoref{mainthm}, we obtain the following result, which appears as Conjecture~18 in \cite{signed_sumsets_elementary}. We use the notation $\rho_\pm(G, m, r)$ defined in \cite{signed_sumsets_elementary}.
\begin{restatable}[{\cite[Conjecture~18]{signed_sumsets_elementary}}]{theorem}{bajnokmatzkeconjecture} \label{bajnokmatzkeconjecture}
	Let $p > 2$ be a prime number, and let $c$ and $v$ be integers with $0 \leq c \leq p-1$ and $1 \leq v \leq p$. Let $m = cp + v$.
	\begin{enumerate}[(a)]
		\item If $1 \leq c \leq (p-3)/2$, then \[\rho_\pm((\Z / p\Z)^2, m, 2) = (2c+1)p.\]
		\item If $c = (p-1)/2$ and $v \leq (p-1)/2$, then \[\rho_\pm((\Z / p\Z)^2, m, 2) = p^2 - 1.\]
	\end{enumerate}
\end{restatable}

\par  In \autoref{sectioncyclic}, we will prove \autoref{cyclic}. In \autoref{sectionupper}, we will prove \autoref{upperbound}. In \Cref{sectionoutline,sectionindependence,sectionsmall,sectionmainthm}, we will prove \autoref{mainthm}. Finally, in \autoref{bajnokmatzkeproof}, we will prove \autoref{bajnokmatzkeconjecture}.

\section{The cyclic case}\label{sectioncyclic}
The goal of this section is to prove \autoref{cyclic}, which computes $\rho_G^-(r)$ in the case that $G$ is a finite cyclic group. The proof closely follows that of \cite[Theorem~4]{signed_sumsets}, though it should be noted that \autoref{cyclic} does not follow directly from \cite[Theorem~4]{signed_sumsets} due to differences in the definitions of $2_{\pm} A$ and $A - A$.
\cyclic*
\begin{proof}[Proof of \autoref{cyclic}]
	By \autoref{sumsets_arbitrary}, we have \[\rho_G^-(r) \geq \mu_G(r, r) = \min_{d \in D(N)} d \left( 2\left\lceil\frac{r}{d}\right\rceil  - 1\right)\] so it remains to show that \[\rho_G^-(r) \leq \min_{d \in D(N)} d \left( 2\left\lceil\frac{r}{d}\right\rceil  - 1\right).\]
	\par It suffices to show that \[\rho_G^-(r) \leq d \left( 2\left\lceil\frac{r}{d}\right\rceil  - 1\right)\] for each $d \in D(N)$.
	For this, we will construct a set $A \subseteq G$ with $|A| \geq r$ and \[|A - A| \leq d \left( 2\left\lceil\frac{r}{d}\right\rceil  - 1\right).\] Let $H$ be the subgroup of $G$ of order $d$, and let $x$ be a generator for $G / H$. Take $A$ to be the ``coset arithmetic progression'' \[A = \bigcup_{i=0}^{\left\lceil r/d\right\rceil - 1} (H + i x).\] We compute \[A - A = \bigcup_{i=1 - \left\lceil r/d\right\rceil }^{\left\lceil r/d\right\rceil - 1} (H + i x),\] so $|A| = d \lceil r/d \rceil \geq r$ and \[|A - A| \leq d \left( 2\left\lceil\frac{r}{d}\right\rceil  - 1\right)\] as desired.
\end{proof}

\begin{remark}
	By comparing the expressions in \autoref{sumsets_arbitrary} and \autoref{cyclic}, we see that $\rho_G^-(r) = \rho_G^+(r) = \mu_G(r, r)$ for $1 \leq r \leq N$ if $G = \Z / N\Z$ is a finite cyclic group.
\end{remark}

\section{\texorpdfstring{An upper bound on $\rho_G^-(r)$}{An upper bound}} \label{sectionupper}
We shall now restate and prove \autoref{upperbound}. The proof very closely follows that of \cite[Theorem~5]{signed_sumsets}.

\upperbound*

\begin{proof}
	\par It suffices to show that \[\rho_G^-(r) \leq d \left(2 \left \lceil \frac{r}{d} \right \rceil - 1\right)\] for each $d \in D(N, e, r)$. For this, we will construct a set $A \subseteq G$ with $|A| \geq r$ and \[|A - A| \leq d \left( 2\left\lceil\frac{r}{d}\right\rceil  - 1\right).\]  Write $d = d_1 d_2$ for $d_1 \in D(N / e)$, $d_2 \in D(e)$, and $d_1 e \geq r$.
	\par By the structure theorem for finitely generated abelian groups, the group $G$ is isomorphic to a direct product $H \times (\Z/e\Z)$ for some abelian group $H$ with $|H| = N/e$. Since $d_1 \in D(N / e)$, we can find a subgroup $A_1 \subseteq H$ with $|A_1| = d_1$. Let $s = \lceil r / d_1 \rceil$. Then $s \leq e$, so by \autoref{cyclic} there is a subset $A_2 \subseteq \Z / e \Z$ with $|A_2| = s$ and \[|A_2 - A_2| \leq d_2 \left(2 \left\lceil\frac{s}{d_2}\right\rceil - 1\right).\]
	Take $A = A_1 \times A_2 \subseteq H \times (\Z / e \Z) \cong G$. Then $|A| = d_1s = d_1\lceil r / d_1 \rceil \geq r$ and
	\begin{align*}
		|A - A| &=|(A_1 \times A_2) - (A_1 \times A_2)| \\
		&=|(A_1 - A_1) \times (A_2 - A_2)| \\
		&=|A_1 - A_1| |A_2 - A_2| \\
		&\leq d_1  \left(d_2 \left(2 \left\lceil\frac{\lceil r / d_1 \rceil}{d_2}\right\rceil - 1\right)\right) \\
		&= d \left(2 \left \lceil \frac{r}{d} \right \rceil - 1\right)
	\end{align*}
	as desired.
\end{proof}

\section{An outline of the proof of \autoref{mainthm}}\label{sectionoutline}
\Cref{sectionoutline,sectionindependence,sectionsmall,sectionmainthm} of this paper will contain the proof of \autoref{mainthm}, which will proceed in four steps:

\begin{enumerate}
	\item \label{step_upper} We will show that the bound given in \autoref{mainthm} is achieved. That is, we will show that \[\rho_G^-(r) \leq p^t \min\left\{2\left\lceil\frac{r}{p^t}\right\rceil - 1, p\right\}.\]
	\item \label{step_inv} We will show that for $G = (\Z / p\Z)^d$, the quantity $\rho_G^-(r)$ only depends on $r$ and $p$ and not $d$, as long as $d$ is large enough that $\rho_G^-(r)$ is defined (that is, $r \leq p^d$).
	\item \label{step_small} By applying the Cauchy-Davenport~Theorem (\autoref{cauchydavenport}) repeatedly, we will prove \autoref{mainthm} in the case that $r \leq p^2$.
	\item \label{step_large} We will conclude the proof of the theorem by induction on $r$.
\end{enumerate}

We start with the following result, which is step~\eqref{step_upper} above.

\begin{lemma}\label{uppervector}
	With the notation of \autoref{mainthm}, we have\[\rho_G^-(r) \leq p^t \min\left\{ 2\left\lceil\frac{r}{p^t}\right\rceil - 1, p\right\}.\]
\end{lemma}

\begin{proof}
	\par Using the notation of \autoref{upperbound}, we have $N = |G| = p^d$ and $e = \exp G = p$, so \begin{align*}
	D(N, e, r) &= \{d_1 d_2 \mid d_1 \in D(p^{d-1}), d_2 \in D(p), d_1p \geq r\} \\
	&= \{p^t, p^{t+1}, \ldots, p^{d-1}, p^d\}.
	\end{align*}
	By \autoref{upperbound}, we have \begin{align*}\min_{d \in D(N, e, r)} d \left(2 \left\lceil \frac{r}{d} \right \rceil - 1\right) &= \min\left\{p^t \left(2\left\lceil\frac{r}{p^t}\right\rceil - 1\right), p^{t+1}, \ldots, p^{d-1}, p^d\right\} \\
	&= p^t \min\left\{2\left\lceil\frac{r}{p^t}\right\rceil - 1, p\right\},\end{align*}
	as desired.
\end{proof}

\begin{remark}
	The proof of \autoref{uppervector} given above shows that \autoref{mainthm} agrees with the prediction given by \autoref{difference_conjecture}.
\end{remark}
\begin{remark}
	Here is an explicit example of a subset $A \subseteq G$ achieving the bound of \autoref{uppervector}. Put a total order $<$ on $\Z / p\Z$ by identifying it with $\{0, 1, \ldots, p-1\}$ in the usual way. Then, recall that $(\Z / p\Z)^d$ is totally ordered by the \textit{lexicographic order}, which is defined as follows: we say that $x = (x_1, \ldots, x_d)$ precedes $y = (y_1, \ldots, y_d)$ in the lexicographic order if for some $i$ we have $x_i < y_i$ and $x_j = y_j$ for $j < i$. Let $A$ be the set of the smallest $r$ elements of $(\Z / p\Z)^d$ in the lexicographic order. Then one can easily verify that \[|A - A| = p^t \min\left\{ 2\left\lceil\frac{r}{p^t}\right\rceil - 1, p\right\},\] which provides an alternative constructive proof of \autoref{uppervector}. It is worth noting that by \cite[Proposition~3.1]{sumsets_vectors}, the same set $A$ satisfies $|A + A| = \rho_G^+(r)$.
\end{remark}

\section{Independence of dimension}\label{sectionindependence}
The following result is step~\eqref{step_inv} in the proof of \autoref{mainthm}.

\begin{lemma}\label{invariance}
	Let $p$ be a prime and let $d_1 > d_2 \geq 0$ be integers. Let $G = (\Z / p\Z)^{d_1}$ and $H = (\Z / p\Z)^{d_2}$. Then $\rho_{G}^-(r) = \rho_{H}^-(r)$ for $1 \leq r \leq p^{d_2}$. 
\end{lemma}

\begin{proof}
	It suffices to consider the case that $d_1 = d_2 + 1$. Since $H$ embeds in $G$ as a subgroup, we have $\rho_{G}^-(r) \leq \rho_{H}^-(r)$, so it remains to show that $\rho_{H}^-(r) \leq \rho_{G}^-(r)$.
	\par Take a subset $A \subseteq G$ with $|A| = r$ and $|A - A| = \rho_{G}^-(r)$. Considering $G$ as a vector space of dimension $d_1 = d_2  + 1$ over the finite field $\F_p$, there are \[\frac{p^{d_1} - 1}{p - 1} = 1 + p + \cdots + p^{d_2} \geq p^{d_2}\] lines containing $0$ (that is, vector subspaces of dimension $1$) in $G$. On the other hand, there are only \[|A-A| - 1 \leq \rho_{G}^-(r) - 1 \leq \rho_{H}^-(r) - 1 < p^{d_2}\] nonzero elements of $A - A$. Since no two distinct lines in $G$ containing $0$ share a nonzero element, we conclude that there is a line $\ell$ in $G$ such that $\ell \cap (A - A) = \{0\}$.
	\par Considering $H$ as a vector space of dimension $d_2 = d_1 - 1$ over $\F_p$, fix an $\F_p$-linear transformation $\pi : G \to H$ whose kernel is the line $\ell$. Such a transformation $\pi$ exists because \[\dim_{\F_p} \ell + \dim_{\F_p} H = 1 + d_2 = d_1 = \dim_{\F_p} G.\] We claim that the restriction $\pi|_{A}$ is an injection. To show this, take $x, y \in A$ with $\pi(x) = \pi(y)$; we will show that $x = y$. Since $\pi$ is linear, we have $\pi(x - y) = 0$, so $x - y \in \ker \pi = \ell$. Therefore, we have $x - y \in \ell \cap (A - A) = \{0\}$. That is, we have $x = y$, as desired.
	\par Since $\pi|_{A}$ is an injection, we have $|\pi(A)| = |A| = r$, where $\pi(A)$ is the image of $A$ under the map $\pi$. Therefore \[\rho_{H}^-(r) \leq |\pi(A) - \pi(A)| = |\pi(A - A)| \leq |A - A| = \rho_{G}^{-}(r)\] as desired.
\end{proof}

\section{\texorpdfstring{The case $r \leq p^2$}{The case r≤p²}}\label{sectionsmall}
In this section, we show that the statement of \autoref{mainthm} holds when $r \leq p^2$, which is step \eqref{step_small} in the proof of \autoref{mainthm}.

\begin{lemma}\label{smallr}
	Let $p$ be a prime and let $d$ be a nonnegative integer. Let $G$ be the group $(\Z / p\Z)^d$. Then \[\rho_G^-(r) = p^t \min\left\{ 2\left\lceil\frac{r}{p^t}\right\rceil - 1, p\right\}\] for $1 \leq r \leq \min\{p^d, p^2\}$, where $t$ is the unique integer satisfying $p^t < r \leq p^{t+1}$.
\end{lemma}

The following lemma will be instrumental in the proof of \autoref{smallr}.

\begin{restatable}{lemma}{sumlemma}\label{sumlemma2}
	Let $p$ be a prime, and let $m$ and $n$ be integers with $n \geq 1$ and $n+2 \leq m \leq (p-1)/2$. Let $\lambda = (\lambda_1, \ldots, \lambda_m)$ be a sequence of integers with $p \geq \lambda_1 \geq \cdots \geq \lambda_m > 0$ and $\sum_{k=1}^m \lambda_k \geq n p + 1$. Let $\mu = (\mu_1, \ldots, \mu_{2m - 1})$ be a sequence of integers such that $\mu_{i + j - 1} \geq \min\{\lambda_i + \lambda_j - 1, p\}$ for $1 \leq i, j \leq m$. Then \[\sum_{k = 1}^{2m-1} \mu_k \geq (2n + 1)p.\]
\end{restatable}

\begin{proof}
	We defer the proof to \autoref{appendixproof}.
\end{proof}

\begin{proof}[Proof of \autoref{smallr}]
	\setcounter{case}{0}
	\par By \autoref{uppervector}, we have \[\rho_G^-(r) \leq p^t \min\left\{ 2\left\lceil\frac{r}{p^t}\right\rceil - 1, p\right\},\] so it remains to show that \begin{equation}\label{smallrlower}\rho_G^-(r) \geq p^t \min\left\{ 2\left\lceil\frac{r}{p^t}\right\rceil - 1, p\right\}.\end{equation}
	
	\par If $r \leq p$, then this follows directly from \autoref{invariance} and the Cauchy-Davenport~Theorem. Thus, we may assume $r > p$.
	
	\par By \autoref{invariance}, we may assume that $d = 2$, so $G = (\Z / p\Z)^2$. If $p = 2$, then the theorem follows easily from enumerating all possible values of $r$ and all sets $A \subseteq G$, so assume that $p > 2$. Let \[r' = \begin{cases}
	p\left(\left\lceil r/p\right\rceil - 1\right) + 1 & \mbox{if $r \leq p(p-1)/2$} \\
	p(p-1)/2 + 1 & \mbox{if $r > p(p-1)/2$} 
	\end{cases}.\] Since $r \geq r'$, replacing $r$ with $r'$ cannot increase the left-hand side of \eqref{smallrlower}, and it is easy to check that this replacement leaves the right-hand side unchanged. Therefore, we may assume that $r = np + 1$ where $1 \leq n \leq (p-1)/2$. Take a subset $A \subset G$ with $|A| = r$; we will show that \[|A - A| \geq (2n+1)p = p^t \min\left\{ 2\left\lceil\frac{r}{p^t}\right\rceil - 1, p\right\}.\]
	
	\par Identify $G$ with the two-dimensional vector space $\F_p^2$ over the field $\F_p$. We will now count the two-element subsets of $A$ in two ways. By definition, the number of two-element subsets of $A$ is the binomial coefficient $\binom{np+1}{2}$. On the other hand, every two-element subset of $A$ is contained in a unique line (that is, affine subspace of $G$ of dimension $1$), so we can count these subsets according to the lines containing them. This yields \begin{equation}\label{linecounting}\sum_{\ell \subset G} \binom{|A \cap \ell|}{2} = \binom{np+1}{2}\end{equation} where the sum is over all lines $\ell \subset G$. Every line in $G$ is parallel to exactly one line $\ell' \subset G$ containing $0$, so \eqref{linecounting} can be rewritten as \[\sum_{\substack{\ell' \subset G \\ \ell' \ni 0}} \sum_{\substack{\ell \subset G \\ \ell \parallel \ell'}} \binom{|A \cap \ell|}{2} = \binom{np+1}{2}\] where the outer sum is over all lines $\ell' \subset G$ containing $0$, and the inner sum is over all lines $\ell \subset G$ parallel to $\ell'$. Since there are exactly $p+1$ lines in $G$ containing $0$, there is a particular line $\ell_0 \subset G$ containing $0$ such that \[\sum_{\substack{\ell \subset G \\ \ell \parallel \ell_0}} \binom{|A \cap \ell|}{2} \geq \frac{1}{p+1} \binom{np+1}{2}.\] We may assume, by applying an $\F_p$-linear change of coordinates, that $\ell_0$ is the line $\{(0, y) \mid y \in \F_p\} \subset \F_p^2 = G$. For any $x \in \F_p$, define the line \[\ell_{x} = \{(x, y) \mid y \in \F_p\}.\] Then, the lines in $G$ parallel to $\ell_0$ are exactly the lines $\ell_x$ for $x \in \F_p$. Let \[m = \max_{x \in \F_p} |A \cap \ell_x|.\] Since \[\sum_{x \in \F_p} |A \cap \ell_x| = |A| = np+1,\] we have $m \geq \lceil (np+1)/p\rceil = n+1$. We consider three cases, depending on whether $m \geq (p+1)/2$, or $m = n+1$, or $n+2 \leq m \leq (p-1)/2$.
	
	\begin{case}[$m \geq (p+1)/2$]
		Take $x \in \F_p$ such that $|A \cap \ell_x| = m$. Since $\ell_x$ is a translate of $\ell_0$, which is isomorphic as a group to $\Z / p\Z$, the Cauchy-Davenport~Theorem applies to the difference $(A \cap \ell_x) - (A \cap \ell_x) \subseteq \ell_0$, yielding \[|(A - A) \cap \ell_0| \geq |(A \cap \ell_x) - (A \cap \ell_x)| \geq \min\{2m-1, p\} = p.\] (Essentially, we are applying the Cauchy-Davenport~Theorem only to the second coordinates of the elements of $A \cap \ell_x$, which lie in $\Z / p\Z$.) That is, the line $\ell_0$ is a subset of $A - A$.
		\par Now, take \textit{any} line $\ell' \subset G$ containing $0$. There is a line $\ell$ parallel to $\ell'$ such that $|A \cap \ell| \geq \lceil (np+1)/p \rceil = n+1$. Since $\ell$ is a translate of $\ell'$, which is isomorphic as a group to $\Z / p\Z$, the Cauchy-Davenport~Theorem again applies to the difference $(A \cap \ell) - (A \cap \ell) \subseteq \ell'$, yielding \[|(A - A) \cap \ell'| \geq |(A \cap \ell) - (A \cap \ell)| \geq \min\{2(n+1)-1, p\} = 2n+1.\]
		
		\par Since $G \setminus \{0\}$ is equal to the disjoint union \[\bigsqcup_{\substack{\ell' \subset G \\ \ell' \ni 0}}(\ell' \setminus \{0\})\] over all lines $\ell' \subset G$ containing $0$, we conclude
		\begin{align*}
		|A - A| &= 1 + \sum_{\substack{\ell' \subset G \\ \ell' \ni 0}} (|(A - A) \cap \ell'| - 1) \\
		&\geq 1 + (p-1) + p\cdot((2n+1) - 1) \\
		&= (2n+1)p
		\end{align*} which is the desired inequality.
	\end{case}
	
	\begin{case}[$m = n+1$]
		Let $S = \{x \in \F_p \mid |A \cap \ell_x| = n+1\}$ and let $s = |S|$. For each $x \in \F_p \setminus S$ we have $|A \cap \ell_x| \leq n$, so \begin{align*}
		\frac{1}{p+1} \binom{np+1}{2} &\leq \sum_{x \in \F_p} \binom{|A \cap \ell_x|}{2} \\
		&= s \binom{n+1}{2} + \sum_{x \in \F_p \setminus S} \binom{|A \cap \ell_x |}{2} \\
		&\leq s\binom{n+1}{2} + \sum_{x \in \F_p \setminus S} \frac{n-1}{2} |A \cap \ell_x| \\
		&= s \binom{n+1}{2} + \frac{n-1}{2} ((np + 1) - (n+1)s),  
		\end{align*}
		Simplifying this inequality and using the bound $n \leq (p-1)/2$, we obtain \begin{align*}
		s &\geq \frac{p+1-n}{p+1} \cdot \frac{np+1}{n+1} \\
		&\geq \frac{p+1-(p-1)/2}{p+1} \cdot \frac{p(p-1)/2+1}{(p-1)/2+1} \\
		&= \frac{p-1}{2} + \frac{p^2 + 7}{2(p+1)^2} \\
		&> \frac{p-1}{2}.
		\end{align*} Thus $s \geq (p+1)/2$, so by the Cauchy-Davenport~Theorem, we have $|S - S| \geq \min\{2s - 1, p\} = p$, so $S - S = \F_p$.
		\par Now, take any $x \in \F_p$. Since $x \in S - S$, there is $y \in \F_p$ such that $y, x+y \in S$. By the Cauchy-Davenport~Theorem again, we have \[|(A - A) \cap \ell_x| \geq |A \cap \ell_{x+y} - A \cap \ell_y| \geq \min\{2(n+1) -1, p\} = 2n + 1.\] Summing over all $x \in \F_p$ yields \[|A - A| = \sum_{x \in \F_p} |(A - A) \cap \ell_x| \geq (2n+1)p\] as desired.
	\end{case}
	
	\begin{case}[$n+2 \leq m \leq (p-1)/2$]
		For $1 \leq k \leq p$, define
		\begin{align*}
		\Lambda_k &= \{x \in \F_p \mid |A \cap \ell_x| \geq k\} \\ 
		M_k &= \{x \in \F_p \mid |(A - A) \cap \ell_x| \geq k\} \\
		\lambda_k &= |\Lambda_k| \\
		\mu_k &= |M_k|
		\end{align*}
		By definition, we have $p \geq \lambda_1 \geq \cdots \geq \lambda_m > 0$ and $p \geq \mu_1 \geq \cdots \geq \mu_p \geq 0$. We have \[\sum_{k=1}^m \lambda_k = \sum_{x \in \F_p} |A \cap \ell_x| = |A| = ap+1\] because each line $\ell_x$ contributes exactly $|A \cap \ell_x|$ to the sum. Similarly \[\sum_{k=1}^p \mu_k = \sum_{x \in \F_p} |(A-A) \cap \ell_{x}| = |A - A|.\]
		
		\par We claim that $M_{i+j-1} \supseteq \Lambda_i - \Lambda_j$ for $1 \leq i, j \leq m$. To show this, take $x_1 \in \Lambda_i$ and $x_2 \in \Lambda_j$; we will show that $x_1 - x_2 \in M_{i+j-1}$. By the Cauchy-Davenport~Theorem, we have
		\begin{align*}
		|(A - A) \cap \ell_{x_1 - x_2}| &\geq |A \cap \ell_{x_1} - A \cap \ell_{x_2}| \\
		&\geq \min\{|A \cap \ell_{x_1}| + |A \cap \ell_{x_2}| - 1, p\} \\
		&\geq \min\{i+j-1, p\} \\
		&= i+j-1
		\end{align*}
		where the last equality follows from the bound $i, j \leq m \leq (p-1)/2$. That is, we have $x_1 - x_2 \in M_{i+j-1}$, as desired.
		\par By the Cauchy-Davenport~Theorem again, we conclude \begin{equation}\label{lambdaineq}\mu_{i+j-1} = |M_{i + j - 1}| \geq |\Lambda_i - \Lambda_j| \geq \min\{\lambda_i + \lambda_j - 1, p\}\end{equation} for $1 \leq i, j \leq m$.
		\par Therefore, the conditions of \autoref{sumlemma2} are satisfied, so
		\[
			|A - A| = \sum_{k=1}^p \mu_k \geq (2n+1)p
		\]
		as desired.
		\qedhere
	\end{case}
\end{proof}

\section{Completing the proof of \autoref{mainthm}} \label{sectionmainthm}

\par Before proceeding to the proof of \autoref{mainthm}, we prove a general lemma about sets in vector spaces over finite fields.
	
\begin{lemma}\label{intersectionlemma}
	Let $p$ be a prime and let $m$ be an integer. Let $G$ be a vector space over the field $\F_p$ of dimension $d \geq 3$, and let $S$ be a subset of $G$ such that \[|S \cap H| \geq mp^{d-2}\] for each vector hyperplane $H$ (that is, vector subspace of dimension $d - 1$) in $G$. Then $|S| \geq mp^{d-1}$.
\end{lemma}
\begin{proof}[Proof of \autoref{intersectionlemma}]
	 Assume for the sake of contradiction that $|S| < mp^{d - 1}$. We first claim that there is a $(d-2)$-dimensional vector subspace $V_0 \subset G$ with $|S \cap V_0| \leq mp^{d-3}$. To show this, take a $(d-2)$-dimensional vector subspace $V \subset G$ uniformly at random. It is clear that $V$ has $p^{d-2} - 1$ nonzero elements, that $G$ has $p^d-1$ nonzero elements, and that each nonzero element of $G$ is in $V$ with equal probability. Therefore, the probability that $x \in V$ for a fixed $x \in G \setminus \{0\}$ is \[\frac{p^{d-2} - 1}{p^{d} - 1}.\] Clearly, the probability that $0 \in V$ is $1$. Therefore, by the linearity of expectation, the expected value of $|S \cap V|$ is given by
	 \begin{align*}
	 \mathbb{E}[|S \cap V|] &= 1 + (|S| - 1) \frac{p^{d-2} - 1}{p^{d} - 1} \\
	 &< 1 + (mp^{d-1} - 1) \frac{p^{d-2} - 1}{p^{d} - 1} \\
	 &= m p^{d-3} + \frac{(p^2 - 1)(p-m)p^{d-3}}{p^{d} - 1} \\
	 &< m p^{d-3} + 1.
	 \end{align*}
	 Since $mp^{d-3}$ is an integer, we conclude that there is a particular $(d-2)$-dimensional vector subspace $V_0 \subset G$ with $|S \cap V_0| \leq mp^{d-3}$.
	 
	 \par Finally, consider the integer $N$ defined by the sum \[N = \sum_{H} |S \cap H|\] where $H$ ranges over all vector hyperplanes with $V_0 \subset H \subset G$. Such hyperplanes $H$ are in bijection with lines through the origin in the two-dimensional quotient space $G / V_0$, so there are $p+1$ of them. Therefore, by the assumption of the theorem, we have \[N \geq \sum_{H} mp^{d-2} = (p+1)mp^{d-2}.\] On the other hand, the sum defining $N$ counts every element of $S \setminus V_0$ once and every element of $S \cap V_0$ exactly $p+1$ times, so \[N = |S| + p |S \cap V_0|.\] Therefore, we have \[|S| = N - p|S \cap V_0| \geq (p+1) mp^{d-2} - p \cdot mp^{d-3} = mp^{d-1},\] which contradicts our assumption that $|S| < mp^{d-1}$.
\end{proof}

\par We are now ready to restate and prove \autoref{mainthm}.

\mainthm*

\begin{proof}
	\par We proceed by induction on $r$. If $t < 2$, then the result follows from \autoref{smallr}, so we may assume $t \geq 2$. By \autoref{invariance}, we may also assume that $d = t+1$. Let $m = \min\{2 \left\lceil r/p^t \right \rceil - 1, p\}$. We wish to show that $\rho_G^-(r) = m p^t$. By \autoref{uppervector}, we have $\rho_G^-(r) \leq mp^t$, so it remains to show that $\rho_G^-(r) \geq m p^t$. Let $A$ be a subset of $G$ with $|A| = r$; we will show that $|A - A| \geq mp^t$.
	
	\par Consider $G$ as a vector space of dimension $d = t + 1 \geq 3$ over $\F_p$. By \autoref{intersectionlemma} applied to $S = A - A$, it suffices to show that $|(A - A) \cap H| \geq mp^{t-1}$ for each vector hyperplane $H \subset G$. For this, note that there are exactly $p$ distinct translates $H + x$, where $x \in G$, and that the entire space $G$ is the disjoint union of these $p$ translates. Therefore, there exists $x_0 \in G$ such that $|A \cap (H + x_0)| \geq \lceil r / p \rceil$. By the inductive hypothesis, \[|(A - A) \cap H| \geq |(A \cap (H + x_0)) - (A \cap (H + x_0))| \geq \rho_{H}^-(\lceil r/p \rceil) = mp^{t-1}\] as desired.
\end{proof}

\section{Applications to signed sumsets} \label{bajnokmatzkeproof}
In this section, we prove \autoref{bajnokmatzkeconjecture}. In particular, we will show that it is a consequence of the following more general result. The notations $\rho_\pm(G, m, r)$ and $r_\pm A$ used in this section are defined in \cite{signed_sumsets_elementary}.
\begin{lemma}\label{signedsumsetlemma}
	Let $G$ be a finite abelian group of order $N$. Then \[\rho_\pm(G, m, 2) \geq \min\{\rho_G^-(m), \rho_G^-(2m) - 1\}\] for $1 \leq m \leq N/2$.
\end{lemma}
\begin{proof}
	\setcounter{case}{0}
	Let $A \subseteq G$ be a subset with $|A| = m$. We will show that \[2_\pm A \geq \min\{\rho_G^-(m), \rho_G^-(2m) - 1\}.\]
	We consider two cases, depending on whether or not $A \cap (-A) = \emptyset$.
	\begin{case}[$A \cap (-A) \neq \emptyset$]
		Choose $x \in A \cap (-A)$. By definition, the signed sumset $2_\pm A$ contains $0 = x + (-x)$ and it contains the difference of any two distinct elements of $A$. Therefore, we have $A - A \subseteq 2_\pm A$. It follows that \[|2_\pm A| \geq |A - A| \geq \rho_G^-(m) \geq \min\{\rho_G^-(m), \rho_G^-(2m) - 1\},\] as desired.
	\end{case}
	\begin{case}[$A \cap (-A) = \emptyset$]
		Let $B = A \cup (-A)$. Then $|B| = 2|A|$. By definition, the signed sumset $2_\pm A$ contains $(B - B) \setminus \{0\}$, so \begin{align*}
		|2_\pm A| &\geq |B - B| - 1 \\
		&\geq \rho_G^-(2m)-1 \\
		&\geq \min\{\rho_G^-(m), \rho_G^-(2m) - 1\},\end{align*} as desired. \qedhere
	\end{case}
\end{proof}

Now, we shall restate and prove \autoref{bajnokmatzkeconjecture}.
\bajnokmatzkeconjecture*

\begin{proof}
	\begin{enumerate}[(a)]
		\item By \autoref{signedsumsetlemma} and \autoref{mainthm}, we have \begin{align*}
		\rho_\pm((\Z / p\Z)^2, m, 2) &\geq \min\{\rho_G^-(m), \rho_G^-(2m) - 1\} \\ 
		&= \min\left\{(2c+1)p, \left(4c+2\left\lceil\frac{2v}{p}\right\rceil + 1\right)p - 1\right\} \\
		&= (2c+1)p.
		\end{align*} The reverse inequality $\rho_\pm((\Z / p\Z)^2, m, 2) \leq (2c+1)p$ follows from \cite[Theorem~5]{signed_sumsets}.
		\item By \autoref{signedsumsetlemma} and \autoref{mainthm}, we have \begin{align*}
		\rho_\pm((\Z / p\Z)^2, m, 2) &\geq \min\{\rho_G^-(m), \rho_G^-(2m) - 1\} \\ 
		&= \min\{p^2, p^2 - 1\} \\
		&= p^2 - 1.
		\end{align*}
		The reverse inequality $\rho_\pm((\Z / p\Z)^2, m, 2) \leq p^2 - 1$ follows from \cite[Proposition~8]{signed_sumsets}.
	\end{enumerate}
\end{proof}

\begin{appendices}
\numberwithin{lemma}{section}

\section{Proof of \autoref{sumlemma2}} \label{appendixproof}
In this appendix, we prove \autoref{sumlemma2}. The following lemma is essential to our proof of \autoref{sumlemma2}.

\begin{lemma}\label{sumlemma}
	Let $m > 1$, and let $\lambda = (\lambda_1, \ldots, \lambda_m)$ be a sequence of integers with $\lambda_1 \geq \cdots \geq \lambda_m > 0$ and $\lambda_1 > 1$. Define the sequence $\mu = (\mu_1, \ldots, \mu_{2m - 1})$ by \[\mu_{k} = \max_{k = i + j - 1} (\lambda_i + \lambda_j - 1)\] for $1 \leq k \leq 2m-1$, where the maximum is over all $1 \leq i, j \leq m$ with $k = i + j - 1$. Then \[\sum_{k=1}^{2m-1} \mu_k \geq 3\left(\sum_{k=1}^m \lambda_k\right) - 3.\]
\end{lemma}

\begin{proof}
	Let \[F(\lambda) = \{(x, y) \in \Z^2 \mid 0 \leq y \leq m-1, 0 \leq x \leq \lambda_{y+1} - 1\} \subset \Z^2\] be the Ferrers diagram of $\lambda$; that is, a set with $m$ rows of points where the $k$th row from the bottom contains $\lambda_k$ points for $1 \leq k \leq m$. Similarly, let \[F(\mu) = \{(x, y) \in \Z^2 \mid 0 \leq y \leq 2m-2, 0 \leq x \leq \mu_{y+1} - 1\} \subset \Z^2\] be the Ferrers diagram of $\mu$.
	
	\par We claim that $F(\mu)$ contains the sumset $F(\lambda) + F(\lambda)$. To show this, take two elements $(x, y)$ and $(x', y')$ in $F(\lambda)$; we wish to show that $(x+x', y+y') \in F(\mu)$. By the definition of $F(\lambda)$ we have 
	\begin{align*}
	0 &\leq y + y' \leq (m-1) + (m-1) = 2m-2 \\
	0 &\leq x + x' \leq (\lambda_{y+ 1} - 1) + (\lambda_{y' + 1} - 1) \leq \mu_{y + y' + 1} - 1
	\end{align*}
	so $(x+x', y + y') \in F(\mu)$ as desired.
	
	\par By assumption, both $m > 1$ and $\lambda_1 > 1$, so $F(\lambda)$ contains the three non-collinear points $(0, 0)$, $(1, 0)$, and $(0, 1)$. Therefore, by Freiman's~dimension~lemma \cite[Theorem~5.20]{tao}, \[\sum_{k=1}^{2m-1} \mu_k = |F(\mu)|\geq |F(\lambda) + F(\lambda)| \geq 3|F(\lambda)| - 3 = 3\left(\sum_{k=1}^m \lambda_k\right) - 3\] as desired.
\end{proof}

We shall now restate and prove \autoref{sumlemma2}.

\sumlemma*

\begin{proof}[Proof of \autoref{sumlemma2}]
	\setcounter{case}{0}
	We may assume that \[\mu_k = \max_{k = i + j - 1}\min \{\lambda_i + \lambda_j - 1, p\}\] for all $k$. Let $h$ be the maximum value of $i + j - 1$ over all integers $1 \leq i, j \leq m$ with $\lambda_i + \lambda_j - 1 > p$, or $0$ if no such $i$ and $j$ exist. Then $\mu_k = p$ for $k \leq h$ and $\mu_{i + j - 1} \geq \lambda_i + \lambda_j - 1$ for $1 \leq i, j \leq m$ as long as $i + j - 1 > h$.
	
	Proceed by induction on $m$. We consider three cases, depending on whether $h = 0$ or $h =1$ or $h \geq 2$.
	
	\begin{case}[$h = 0$]
		Then \autoref{sumlemma} applies, so 
		\begin{align*}
		\sum_{k = 1}^{2m-1} \mu_k &\geq 3\left(\sum_{k=1}^{m}\lambda_k\right) - 3 \\
		&\geq 3(np+1) - 3 \\
		&\geq (2n+1)p
		\end{align*}
		as desired.
	\end{case}
	
	\begin{case}[$h = 1$]
		\par First assume $n = 1$ and $m = 3$. Then 
		\begin{align*}
		\sum_{k=1}^{2m-1} \mu_k &= \mu_1 + \mu_2 + \mu_3 + \mu_4 + \mu_5 \\
		&\geq p + (\lambda_1 + \lambda_2 - 1) + (\lambda_1 + \lambda_3 - 1) + (\lambda_2 + \lambda_3 - 1) + 1 \\
		&\geq p + 2(\lambda_1 + \lambda_2 + \lambda_3) - 2 \\
		&\geq p + 2(p+1)-2 \\
		&= 3p
		\end{align*}
		as desired.
		\par Next assume $n = 1$ and $m \geq 4$. The assumption that $h = 1$ implies that $2 \lambda_1 - 1 > p$, so $\lambda_1 > (p+1)/2$. Therefore $\mu_k \geq \lambda_1 + \lambda_k - 1 > (p+1)/2$ for $1 < k < m$ and $\mu_k \geq \lambda_m + \lambda_{k-m+1} - 1 \geq \lambda_{k-m+1}$ for $k \geq m$, so 
		\begin{align*}
		\sum_{k=1}^{2m-1} \mu_k &> p + \sum_{k=2}^{m-1} \frac{p+1}{2} +\sum_{k=m}^{2m-1} \lambda_{k - m + 1} \\
		&= p + (m-2)\frac{p+1}{2} + (np+1) \\
		&> 3p
		\end{align*} as desired.
		\par It remains to consider the case that $n  \geq 2$. Because $h = 1$, \autoref{sumlemma} applies to the sequences $(\lambda_1, \cdots, \lambda_m)$ and $(2p-1, \mu_2, \cdots, \mu_{2m -1})$. Therefore
		\begin{align*}
		\sum_{k=1}^{2m-1} \mu_k &= p + \sum_{k=2}^{2m-1} \mu_k \\
		&= -p + 1 + \left(2p - 1 + \sum_{k=2}^{2m-1} \mu_k\right) \\
		&\geq -p + 1 + 3(np + 1) - 3 \\
		&\geq (2n+1)p
		\end{align*}
		as desired.
	\end{case}

	\begin{case}[$h \geq 2$]
		Define the sequence $\lambda' = (\lambda'_1, \ldots, \lambda'_{m-1})$ by $\lambda'_k = \lambda_{k+1}$ for $1 \leq k \leq m-1$. Then, define $\mu' = (\mu'_1, \ldots, \mu'_{2m-3})$ by \[\mu'_{k} = \max_{k = i + j - 1} \min\{\lambda'_i + \lambda'_j - 1, p\}\] for $1 \leq k \leq 2m-3$, where the maximum is over all $1 \leq i, j \leq m-1$ with $k = i + j - 1$.
		We have \[\sum_{k=1}^{m-1} \lambda'_k = \left(\sum_{k=1}^{m} \lambda_k\right) - \lambda_1 \geq (n-1)p + 1,\] so by the inductive hypothesis we have \[\sum_{k=1}^{2m-3} \mu'_k \geq (2n - 1)p.\] On the other hand, we have \[\mu_{k+2} =\max_{k+2 = i + j - 1} (\lambda_i + \lambda_j - 1) \geq \max_{k = i + j - 1} (\lambda'_i + \lambda'_j - 1) = \mu'_k\] for $1 \leq k \leq 2m-3$, where the inequality follows from replacing $(i, j)$ with $(i-1, j-1)$. Therefore \[
		\sum_{k=1}^{2m-1} \mu_k = 2p + \sum_{k=1}^{2m-3} \mu'_{k} \\
		\geq (2n + 1)p
		\]
		as desired.
		\qedhere
	\end{case}
\end{proof}
\end{appendices}
\section*{Acknowledgments}
This research was conducted under the supervision of Joseph Gallian at the University of Minnesota Duluth REU, funded by NSF Grant 1358659 and NSA Grant H98230-13-1-0273. The author thanks Joseph Gallian, who ran the REU, brought this question to his attention, and provided helpful comments on the manuscript. He thanks his advisors Levent Alpoge and Benjamin Gunby for valuable discussions and advice.
\bibliography{paper_differences}{}

\begin{thebibliography}{10}

\bibitem{signed_sumsets}
B{\'{e}}la Bajnok and Ryan Matzke.
\newblock The minimum size of signed sumsets.
\newblock {\em Electr. J. Comb.}, 22(2):P2.50, 2015.

\bibitem{signed_sumsets_elementary}
B{\'{e}}la Bajnok and Ryan Matzke.
\newblock On the minimum size of signed sumsets in elementary abelian groups.
\newblock {\em Journal of Number Theory}, 159:384 -- 401, 2016.

\bibitem{bukh_dilates}
Boris Bukh.
\newblock Sums of dilates.
\newblock {\em Combinatorics, Probability and Computing}, 17(05):627--639,
  2008.

\bibitem{cauchy}
Augustin-Louis Cauchy.
\newblock Recherches sur les nombres.
\newblock In {\em Oeuvres compl{\`{e}}tes}, volume~1, pages 39--63. Cambridge
  University Press, 2009.
\newblock Cambridge Books Online.

\bibitem{davenport}
Harold Davenport.
\newblock On the addition of residue classes.
\newblock {\em Journal of the London Mathematical Society}, 1(1):30--32, 1935.

\bibitem{sumsets_vectors}
Shalom Eliahou and Michel Kervaire.
\newblock Sumsets in vector spaces over finite fields.
\newblock {\em J. Number Theory}, 71(1):12--39, 1998.

\bibitem{sumsets_infinite}
Shalom Eliahou and Michel Kervaire.
\newblock Minimal sumsets in infinite abelian groups.
\newblock {\em Journal of Algebra}, 287(2):449 -- 457, 2005.

\bibitem{kneser}
Martin Kneser.
\newblock Absch{\"a}tzung der asymptotischen dichte von summenmengen.
\newblock {\em Mathematische Zeitschrift}, 58(1):459--484, 1953.

\bibitem{sumsets_solvable}
Alain Plagne.
\newblock Optimally small sumsets in groups. {I}. {T}he supersmall sumsets
  property, the {$\mu_G^{(k)}$} and the {$\nu_G^{(k)}$} functions.
\newblock {\em Unif. Distrib. Theory}, 1(1):27--44, 2006.

\bibitem{plagne_dilates}
Alain Plagne.
\newblock Sums of dilates in groups of prime order.
\newblock {\em Combinatorics, Probability and Computing}, 20(06):867--873,
  2011.

\bibitem{pontiveros_dilates}
Gonzalo~Fiz Pontiveros.
\newblock Sums of dilates in {$\mathbb{Z}_p$}.
\newblock {\em Combinatorics, Probability and Computing}, 22(02):282--293,
  2013.

\bibitem{tao}
Terence Tao and Van~H Vu.
\newblock {\em Additive combinatorics}, volume 105.
\newblock Cambridge University Press, 2006.

\end{thebibliography}
\bibliographystyle{plain}
\end{document}